\documentclass[12pt]{amsart}
\usepackage{enumerate}
\usepackage{bbm}
\usepackage{graphicx}
\usepackage{mathrsfs}
\usepackage{amsmath, amssymb, amsthm, latexsym}
\usepackage{amssymb,amscd,amsmath}
\usepackage{latexsym}
\usepackage{MnSymbol}
\usepackage{enumitem}
\usepackage[center]{caption}
\usepackage{tikz}
\usepackage[all]{xy}
\usepackage{amsxtra,mathtools,stmaryrd}
\newcounter{braid}
\newcounter{strands}

\DeclareMathAlphabet{\bsf}{OT1}{cmss}{bx}{n}

\pgfkeyssetvalue{/tikz/braid height}{1cm}
\pgfkeyssetvalue{/tikz/braid width}{1cm}
\pgfkeyssetvalue{/tikz/braid start}{(0,0)}
\pgfkeyssetvalue{/tikz/braid colour}{black}
\pgfkeys{/tikz/strands/.code={\setcounter{strands}{#1}}}

\makeatletter
\def\cross{%
  \@ifnextchar^{\message{Got sup}\cross@sup}{\cross@sub}}

\def\cross@sup^#1_#2{\render@cross{#2}{#1}}

\def\cross@sub_#1{\@ifnextchar^{\cross@@sub{#1}}{\render@cross{#1}{1}}}

\def\cross@@sub#1^#2{\render@cross{#1}{#2}}

\def\render@cross#1#2{
  \def\strand{#1}
  \def\crossing{#2}
  \pgfmathsetmacro{\cross@y}{-\value{braid}*\braid@h}
  \pgfmathtruncatemacro{\nextstrand}{#1+1}
  \foreach \thread in {1,...,\value{strands}}
  {
    \pgfmathsetmacro{\strand@x}{\thread * \braid@w}
    \ifnum\thread=\strand
    \pgfmathsetmacro{\over@x}{\strand * \braid@w + .5*(1 - \crossing) * \braid@w}
    \pgfmathsetmacro{\under@x}{\strand * \braid@w + .5*(1 + \crossing) * \braid@w}
    \draw[braid] \pgfkeysvalueof{/tikz/braid start} +(\under@x pt,\cross@y pt) to[out=-90,in=90] +(\over@x pt,\cross@y pt -\braid@h);
    \draw[braid] \pgfkeysvalueof{/tikz/braid start} +(\over@x pt,\cross@y pt) to[out=-90,in=90] +(\under@x pt,\cross@y pt -\braid@h);
    \else
    \ifnum\thread=\nextstrand
    \else
     \draw[braid] \pgfkeysvalueof{/tikz/braid start} ++(\strand@x pt,\cross@y pt) -- ++(0,-\braid@h);
    \fi
   \fi
  }
  \stepcounter{braid}
}

\tikzset{braid/.style={double=\pgfkeysvalueof{/tikz/braid colour},double distance=1pt,line width=2pt,white}}

\newcommand{\Tau}{\mathrm{T}}

\newcommand{\braid}[2][]{%
  \begingroup
  \pgfkeys{/tikz/strands=2}
  \tikzset{#1}
  \pgfkeysgetvalue{/tikz/braid width}{\braid@w}
  \pgfkeysgetvalue{/tikz/braid height}{\braid@h}
  \setcounter{braid}{0}
  \let\sigma=\cross
  #2
  \endgroup
}
\makeatother

\input xypic
\newtheorem{theorem}{Theorem}
\newtheorem{proposition}[theorem]{Proposition}


\def\Z{\mathbb{Z}}

\def\C{\mathbb{C}}

\def\R{\mathbb{R}}
\def\C{\mathbb{C}}

\def\N{\mathbb{N}}

\def\F{\mathbb{F}}

\def\qed{\hfill$\square$\medskip}

\def\Zpk{\mathbb{Z}/p^{k}}
\def\Zpk1{\mathbb{Z}/p^{k-1}}

\newcommand{\rref}[1]{(\ref{#1})}

\newcommand{\beg}[2]{\begin{equation}\label{#1}#2\end{equation}}
\def\r{\rightarrow}

\def\F{\mathbb{F}}

\def\sl2{\widetilde{SL_{2}(\Z)}}

\author{Po Hu, Igor Kriz, Petr Somberg, and Foling Zou}
\title[Equivariant Steenrod algebra]{The $\Z/p$-equivariant dual Steenrod algebra for an odd prime $p$}
\thanks{Hu acknowledges the support of NSF grant DMS 2301520.
Kriz acknowledges the support of a Simons Collaboration Grant. Somberg acknowledges support
from the grant GACR 19-28628X}

\begin{document}

\maketitle

\begin{abstract}
We completely calculate the $RO(\Z/p)$-graded coefficients $H\underline{\Z/p}_\star H\underline{\Z/p}$
for the constant Mackey functor $\underline{\Z/p}$.
\end{abstract}

\section{Introduction}\label{intro}

In \cite{hk}, Hu and Kriz computed the Hopf algebroid 
\beg{ez2steen}{(H\underline{\Z/2}_\star,
H\underline{\Z/2}_\star H\underline{\Z/2})}
where $\underline{\Z/2}$ denotes the constant $\Z/2$-Mackey functor and the subscript $?_\star$ denotes
$RO(G)$-graded coefficients. (We direct the reader
to \cite{dress,lewis,li,sk} for general preliminaries on Mackey functors, and to \cite{lms,lms1} for their
relationship with ordinary equivariant generalized (co)homology theories.) There are several remarkable
features of the Hopf algebroid \rref{ez2steen}, which played an important role
in the calculation of coefficients of Real-oriented spectra in \cite{hk}. 
For example, the morphism ring of \rref{ez2steen} is a free module
over the object ring (in particular, the Hopf algebroid is ``flat," see \cite{ravenel}). Another remarkable
fact is that the Hopf algebroid \rref{ez2steen} is closely related to the motivic
dual Steenrod algebra Hopf algebroid \cite{voev,hko}. In particular, if we denote by $\alpha$ the real
sign representation, there are generator classes $\xi_i$ in degrees $(2^i-1)(1+\alpha)$ and 
$\tau_i$ in degrees $(2^i-1)(1+\alpha)+1$. A key point in proving this is the existence of the 
infinite complex projective
space with $\Z/2$-action by complex conjugation.

\vspace{3mm}
When trying to generalize this calculation to a prime $p>2$, i.e. to compute
\beg{ezpsteen0}{(H\underline{\Z/p}_\star,
H\underline{\Z/p}_\star H\underline{\Z/p}) \;\text{for $p>2$ prime},}
the first difficulty one encounters is that no $p>2$ generalization of complex conjugation on $\C P^\infty$
presents itself. Additionally, applications of this calculation were not yet known. For those reasons, 
likely, the question
remained open for more than 20 years.

However, recently $\Z/p$-equivariant cohomological operations at $p>2$ have become of interest, for example due to 
questions of $p>2$ analogues of certain computations by Hill, Hopkins, and Ravenel in their solution
to the Kervaire invariant $1$ problem \cite{hhr,  hhrodd}. In fact, Sankar and Wilson
\cite{sw} made progress on the calculation of  \rref{ezpsteen}, in particular providing
a complete decomposition of the spectrum $
H\underline{\Z/p}\wedge H\underline{\Z/p}$ and thereby proving that the
morphism ring \rref{ezpsteen} is {\em not} a flat module over the object ring. Due to this non-flatness,
\rref{ezpsteen0} is slightly the wrong notation. One should write, instead,
\beg{ezpsteen}{(\underbrace{H\underline{\Z/p}\wedge\dots\wedge H\underline{\Z/p}}_{n})_\star
)_{n\in\N_0} \;\text{for $p>2$ prime},}

\vspace{3mm}
The goal of the present paper is to complete the calculation of the structure of \rref{ezpsteen}. 
Due, again, the non-flatness, the algebraic structure on \rref{ezpsteen} is more general than a Hopf algebroid. Let $\Xi$ be the 
{\em cyclic
category}, whose object set is $\N_0$ and morphisms $m\r n$ are maps $\{0,\dots,m\}\r\{0,\dots,n\}$ which preserve the 
non-strict cyclic ordering. Then \rref{ezpsteen} is a functor from $\Xi$ to graded-commutative rings. (More comments on
graded commutativity will be given below.) We shall call such structure a {\em Hopf pluri-algebroid}. We note that
while we shall determine, (at least inductively), all the structure maps of the Hopf pluri-algebroid \rref{ezpsteen}, the formal structure does not
play a heavy role in our discussion.

\vspace{3mm}

The authors note that while the present paper is largely self-contained, comparing with the results of \cite{sw} 
proved to be a step in the present calculation. We also point out the fact that while we do obtain
complete characterizations of all the structure formulas of the Hopf pluri-algebroid \rref{ezpsteen}, 
not all of them are presented in a nice explicit fashion, but only by recursion
through (explicit) maps into other rings. This, in fact, also occurs in \rref{ez2steen}, where
while the product relations and the coproduct are very elegant, the right unit is only recursively characterized
by comparison with Borel cohomology. In the case of \rref{ezpsteen}, this occurs for more of the structure
formulas. 

\vspace{2mm}
The formulas are, in fact, too complicated to present in the introduction, but we outline here
the basic steps and results, with references to precise
statements in the text. We first need to establish
certain general preliminary facts about $\Z/p$-equivariant homology which are given in Section \ref{prelim}
below. The first important realization is that $\Z/p$-equivariant cohomology with constant coefficients
is periodic under differences of irreducible real $\Z/p$-representations of complex type. Therefore, 
since we are only using the additive structure of the real representation ring, we can reduce the indexing
from $RO(\Z/p)$ to the free abelian group $R$ on $1$ (the real $1$-representation) and $\beta$, which is
a chosen irreducible real representation of complex type. 

Considering the full $R$-grading, on the other
hand, is essential, since the pattern is simpler than if we only considered the $\Z$-grading (not to mention 
the fact that it contains more information). As already remarked, even in the $p=2$ case,
not all the free generators are in integral dimensions. 
As noted in \cite{sw}, the dual $\Z/p$-equivariant Steenrod algebra for $p>2$ is not a free module
over the coefficient, but still, the $R$-graded pattern is a lot easier to describe than the $\Z$-graded pattern.

Additively speaking, for $p>2$, the $R$-graded dual $\Z/p$-equivariant Steenrod algebra 
$A_\star=H\underline{\Z/p}_\star H\underline{\Z/p}$
is a sum of copies of the coefficients of 
$H\underline{\Z/p}$, and the coefficients of another $H\underline{\Z/p}$-module $HM$, which, up
to shift by $\beta-2$, 
is the equivariant cohomology corresponding to the Mackey functor $Q$ which is $0$ on the fixed orbit
and $\Z/p$ (fixed) on the free orbit. In Propositions \ref{p1}, \ref{p2}, we describe the 
ring $H\underline{\Z/p}_\star$ and the module $HM_\star$. 

It is quite interesting that for $p=2$, this $H\underline{\Z/p}$-module $HM$ is,
in fact, an $R$-shift of $H\underline{\Z/2}$. For $p>2$, however, this fails, and in fact infinitely many
higher $\underline{\Z/p}$-$Tor$-groups of $Q$ with itself are non-trivial.

How is it possible for $A_\star$ to be manageable, then? As noted in \cite{sw}, 
the spectrum $H\underline{\Z/p}\wedge H\underline{\Z/p}$ is, in fact not a wedge-sum of $R$-graded
suspensions of $H\underline{\Z/p}$ and $HM$. It turns
out that instead, the $HM$-summands of the coefficients form ``duplexes," i.e. pairs each of which makes up an $H\underline{\Z/p}$-module which we denote by $HT$,
whose smashes over $H\underline{\Z/p}$ are again $R$-shifted copies of $HT$
(Proposition \ref{p44}). We establish this,
and further explain the phenomenon, in Section \ref{prelim} below.

\vspace{2mm}

Now our description of the multiplicative properties of $A_\star$ for $p>2$ must, of course, take into
account the smashing rules of the building blocks $HT$
over $H\underline{\Z/p}$. However, we still need to identify elements which
play the role of ``multiplicative generators." In the present situation, this is facilitated in Section 
\ref{slens} by computing the $H\underline{\Z/p}$-cohomology of the equivariant complex
projective space $\C P_{\Z/p}^\infty$ (Proposition \ref{p33}) and the equivariant infinite lens space 
$B_{\Z/p}(\Z/p)$ (Proposition \ref{pzpzp}).

This can be accomplished using the spectral sequence coming from a filtration of the complete complex
universe by a regular flag, which was also used in \cite{kluf} to give a new more explicit
proof of Hausmann's theorem \cite{haus} on the universality of the equivariant formal group on 
stable complex cobordism. 

The spectral sequences can be completely
solved and the multiplicative structure is completely determined by
comparison with Borel cohomology. One can then use an analog of Milnor's method \cite{milnor}
to construct elements of $A_\star$. It is convenient to do both the cases of the projective space and the lens space,
since in the projective case, the spectral sequence actually collapses, 
thus yielding elements of $\underline{\xi}_n\in A_\star$ of dimension 
$$2p^{n-1}+(p^n-p^{n-1}-1)\beta$$
and $\underline{\theta}_n$ of dimension 
$$2(p^n-1)+(p-1)(p^n-1)\beta.$$
One also obtains analogues of Milnor coproduct relations between these elements by the method of 
\cite{milnor}. One notes that the elements $\underline{\theta}_n$ are of the right ``slope" $2k+\ell\beta$
where $(k:\ell)=(1:p-1)$ and 
in fact, they turn out to generate $H\underline{\Z/p}_\star$-summands of $A_\star$. It is also interesting to note
that for $p>2$, it is impossible to shift the element $\underline{\xi}_n$ to the ``right slope," since the 
values of $2k$ and $\ell$ would not be integers.

To completely understand the role of the elements $\underline{\xi}_n$, and also to construct the
remaining ``generators" of $A_\star$, one solves the flag spectral sequence for $B_{\Z/p}(\Z/p)$. This time,
there are some differentials (although not many), which is the heuristic reason why $A_\star$
is not a free $H\underline{\Z/p}_\star$-module. In fact, following the method of \cite{milnor},
one constructs $\underline{\tau}_n$, $\widehat{\tau}_n$, of degree $|\underline{\xi}_n|+1$, 
$|\underline{\xi}_n|$, respectively, and also an element $\widehat{\xi}_n$ of degree
$|\underline{\xi}_n|-1$. The ``multiplet" of elements  
\beg{emultiplet100}{\underline{\tau}_n, \widehat{\tau}_n,
 \underline{\xi}_n, \widehat{\xi}_n,} 
in fact, ``generates" a copy of $HT_\star$. Again, 
Milnor-style coproduct relations on these elements follow using the same method.

The cohomology of the equivariant lens space, in fact, 
gives rise to one additional set of elements $\underline{\mu}_n$ of degree
$|\underline{\theta}_n|+1$, each of which also generates an $H\underline{\Z/p}_\star$-summand
of $A_\star$. Coproduct relations on this element also follow from the method of \cite{milnor}, 
but this time, there are more terms, so they are hard to write down in an elegant form. 

To construct a complete decomposition of $A_\star$ into a sum of copies of $H\underline{\Z/p}_\star$,
$HT_\star$, one notes that along with the ``multiplet" \rref{emultiplet100}, we also have a ``multiplet"
\beg{emultiplet101}{
\underline{\tau}_n\underline{\xi}_n^{i-1}, \widehat{\tau}_n\underline{\xi}_n^{i-1},
 \underline{\xi}_n^{i}, \widehat{\xi}_n\underline{\xi}_n^{i-1}
}
for $i=1,\dots,p-1$, which ``generates" another copy of $HT_\star$, in a degree predicted
by Proposition \ref{p44}. (This, in fact, means that some of the linear combinations
of those elements will be
divisible by the class $b$ of degree $-\beta$ which comes from the 
inclusion $S^0\r S^\beta$. We will return to this point below.) For now, however,
let us remark that for $i=p$, the element $\underline{\xi}_n^p$ is dependent on $\underline{\theta}_n$,
which is precisely accounted for by the presence of the additional element $\underline{\mu}_n$.

\vspace{2mm}
Tensoring these structures for different $n$ then completely accounts for all elements of $A_\star$, 
coinciding with the result of \cite{sw}. To determine the structure formulas of $A_\star$ completely, we need to 
determine the divisibility of the linear combinations of the elements \rref{emultiplet101} 
by $b$, and to completely describe all the multiplicative relations. This can, in fact, be done by
comparing with the Borel cohomology version $A^{cc}_\star$ of the $\Z/p$-equivariant Steenrod algebra,
similarly as in \cite{hk}. The divisibility is described in Proposition \ref{ppdiv}. The final answer is recorded
in Theorem \ref{tfinal}. We write down explicitly those formulas which are simple enough to state explicitly.
The remaining formulas are determined recursively by the specified map into $A^{cc}_\star$.

\vspace{5mm}
\noindent
{\bf Acknowledgement:} The authors are very indebted to Peter May for comments. The authors are also
thankful to Hood Chatham for
creating the spectralsequences latex package and to Eva Belmont for tutorials on it.

\section{Preliminaries}\label{prelim}

Let $p$ be an odd prime. In this paper, we will work $p$-locally. In other words, when we write $\Z$,
we mean $\Z_{(p)}$. For our present purposes, the ring structure on $RO(\Z/p)$
does not matter. Thus, it can be treated as the 
free abelian group on the irreducible real representations, which are $1$ (the trivial irreducible real 
representation) and $\beta^i$, $i=1,\dots, (p-1)/2$ where $\beta$ is the $1$-dimensional complex
representation where the generator $\gamma$ acts by $\zeta_p$ (since $\beta^i$ is the dual of $\beta^{p-i}$).

However, $p$-local $\Z/p$-equivariant generalized cohomology is periodic with period $\beta^i-\beta^j$, $1\leq i,j<p$. 
(This was also observed by Sankar-Wilson \cite{sw}.)
This follows from the fact that $S(\beta^i)$ is the homotopy cofiber of
$$1-\gamma^i:\Z/p_+\r \Z/p_+.$$
Now we have a homotopy commutative diagram of $\Z/p$-equivariant spectra
$$\diagram
\Z/p_+\rto^{1-\gamma^i}&\Z/p_+\\
\Z/p_+\uto^{Id}\rto^{1-\gamma}&\Z/p_+.\uto_{1+\gamma+\dots+\gamma^{i-1}}
\enddiagram$$
On the other hand, $p$-locally $N_i=1+\gamma+\dots+\gamma^{i-1}$ has a homotopy inverse 
equal to a unit times
$$N_{j}^i=1+\gamma^i+\gamma^{2i}+\dots+\gamma^{(j-1)i}$$
where $j$ is a positive integer such that $ij=kp^2+1$. Indeed, one has
\beg{eunitt}{N_{j}^i\circ N_i=1+\gamma+\dots+\gamma^{p^2k}=pkN_p+1.}
Since we are working $p$-locally, the right hand side induces an isomorphism both in $\Z/p$-equivariant 
and non-equivariant homotopy groups, and thus is a unit in the $\Z/p$-equivariant stable
homotopy category.

Therefore, the reduced cell chain complexes of $S^{\beta^i}$, which are the unreduced suspensions of
$S(\beta^i)$, are also isomorphic, which implies the periodicity by the results of \cite{sk, lewis} (see also \cite{hk1,klu}).
Thus, when discussing ordinary $\Z/p$-homology, the grading can be by elements of $R=\Z\{1,\beta\}$,
without losing information.

\vspace{3mm}
Now the Borel cohomology $H\underline{\Z/p}_\star^c=F(E\Z/p_+,H\underline{\Z/p})_\star$ is complex-oriented.
Its $R$-graded coefficients (given by group cohomology) are given by
$$H\underline{\Z/p}^c_\star=\Z/p[\sigma^{2},\sigma^{-2}][b]\otimes \Lambda[u],$$
where the (homological) dimension of the generators is given by
$$|b|=-\beta,\;|\sigma^{2}|=\beta-2, |u|=-1.$$
The Tate cohomology 
$H_\star^t(\underline{\Z/p})=(\widetilde{E\Z/p}\wedge F(E\Z/p_+,H\underline{\Z/p}))_\star$
(where $\widetilde{X}$ denotes the unreduced suspension of $X$) is given by 
$$H\underline{\Z/p}^t_\star=\Z/p[\sigma^{2},\sigma^{-2}][b, b^{-1}]
\otimes \Lambda[u].$$
(The notation $\sigma^2$ is to connect with the notation of \cite{hk}, where the present question
was treated for $p=2$.) The Borel homology $H^b_\star H\underline{\Z/p} = (E\Z/p_+\wedge
H\underline{\Z/p})_\star$ then is given by
$$ H\underline{\Z/p}^b_\star=\Sigma^{-1}H^t_\star H\underline{\Z/p}/H^c_\star H\underline{\Z/p}.$$
Now similarly as for $p=2$, the $\Z$-graded coefficients imply
$$H\underline{\Z/p}_\star^\phi=\Phi^{\Z/p}H\underline{\Z/p}_\star=(\widetilde{E\Z/p}\wedge 
H\underline{\Z/p})_\star=\Z/p[\sigma^{-2}][b,b^{-1}]\otimes \Lambda[\sigma^{-2}u],$$
the map into $H\underline{\Z/p}^t_\star$ is given by inclusion. 

Composing the inclusion with the quotient map into $\Sigma  H\underline{\Z/p}^b_\star$ then gives
the connecting map of the long exact sequence associated with the fibration
$$ H\underline{\Z/p}^b\r  H\underline{\Z/p}\r  H\underline{\Z/p}^\phi$$
where $E^\phi=E\wedge \widetilde{E\Z/p}$ for a $\Z/p$-equivariant spectrum $E$. 

Regarding the commutation rule of a (not necessarily coherently) $\Z/p$-equivariant ring spectrum,
we note that the commutation rule between a class of degree $1$ and a class of degree $\beta$ is
a matter of convention. The commutation rule between two classes of degree $\beta$ is given
by $x\mapsto -x$ on $S^\beta$, which is equivariantly homotopic to the identity. Thus, if the dimensions
of classes $x,y$ are $k+\ell\beta$, $m+n\beta$, then we can write
$$xy=(-1)^{km}yx.$$

This implies the following

\begin{proposition}\label{p1}
Consider the ring 
$$\Gamma=\Z/p[\sigma^{-2}][b]\otimes\Lambda[\sigma^{-2}u]$$
graded as above. Let $\Gamma^\prime$ be the second local cohomology of $\Gamma$ with respect
to the ideal $(\sigma^{-2},b)$, desuspended by $1$. Then
$$H\underline{\Z/p}_\star=\Gamma\oplus \Gamma^\prime$$
with the abelian ring structure over $\Gamma$.

Additively, therefore, $H\underline{\Z/p}_{k+\ell\beta}$ is $\Z/p$ when 
$0\leq k\leq -2\ell$ or $-2\ell\leq k\leq -2$, and $0$ otherwise.
\end{proposition}
\qed

We shall sometimes call the $\Gamma$ part of the coefficients as {\em the good tail}
and the $\Gamma^\prime$ part {\em the derived tail}.

\begin{figure}
\includegraphics{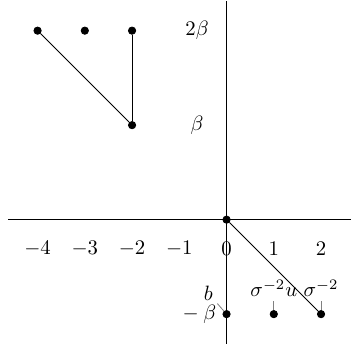}
\caption{The coefficients $H\underline{\Z/p}_\star$}
\end{figure}

\vspace{3mm}
Generally, the $R$-graded coefficients of ring spectra $X$ we will consider will split into a ``good tail"
$X_{\star}^g$ and a ``derived tail" $X_\star^d$. Generally, the coefficient ring will be abelian over the
good tail, and the derived tail will be isomorphic to the second local cohomology of $\Gamma$ with coefficients
in the good tail, desuspended by $1$. This is convenient, since it describes
the multiplication completely. From that point of view, then, we can write
$$\Gamma=H\underline{\Z/p}_\star^g,$$
$$\Gamma^\prime=H\underline{\Z/p}_\star^d=H^2_{(\sigma^{-2},b)}(\Gamma,\Gamma)[-1].$$

\vspace{3mm}

There is another $H\underline{\Z/p}$-module spectrum which will play a role in our calculations.
Consider the short exact sequence of $\Z/p$-Mackey functors
\beg{emq}{0\r Q\r \underline{\Z/p}\r \Phi\r 0}
where $Q$ resp. $\Phi$ are $\Z/p$ on the free resp. fixed orbit, and $0$ on the other orbit. This leads
to a cofibration sequence of (coherent) $H\underline{\Z/p}$-module spectra.
We have  $H\Phi^\phi=H\Phi$, which implies
$$H\Phi_\star=\Z/p[b,b^{-1}].$$
We also note that
\beg{ehqqf}{
HQ^\phi_n=\left\{\begin{array}{ll}\Z/p &\text{for $n\in\{1,2,\dots\}$}\\
0 &\text{else.}
\end{array}\right.
}
The following follows from the long exact sequence on $RO(\Z/p)$-graded coefficients associated with \rref{emq}.

\begin{proposition}\label{p2}
The group $HQ_{k+\ell\beta}$ is $\Z/p$ when $1\leq k\leq -2\ell$ or $-2\ell\leq k\leq -1$ and $0$ otherwise. 
From a $H\underline{\Z/p}_\star$-module point of view, the $\ell<0$ part of the
coefficients (the {\em good tail}) behaves as an ideal in $\Gamma$ and
the $\ell>0$ part {\em the derived tail} behaves as a quotient of $\Gamma^\prime$. Additionally, we have,
again,
$$H_\star^d=H^2_{(b,\sigma^{-2})}(H\underline{\Z/p}_\star^g,HQ_\star^g)[-1].$$
\end{proposition}
\qed

\begin{figure}
\includegraphics{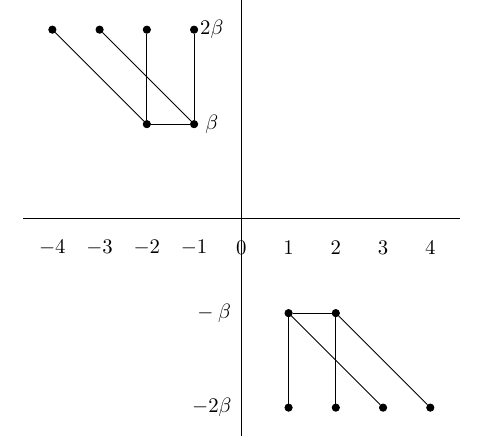}
\caption{The coefficients $HQ_\star$}
\end{figure}

We also put 
$$HM=\Sigma^{\beta-2}HQ$$
for indexing purposes, since the element in degree $0$
is closest to playing the role of the ``generator" of $HM_\star$. (The element
in degree $-1$ is its multiple in Borel cohomology.)

\begin{figure}


\includegraphics{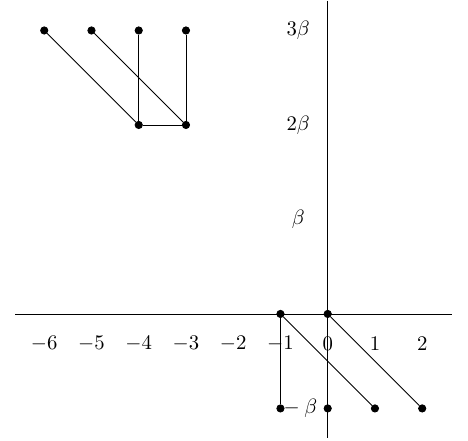}

\caption{The coefficients $HM_\star$}

\end{figure}

Now there is more to the multiplicative structure, however. We are interested in calculating
\beg{ecococo1}{HQ\wedge_{H\underline{\Z/p}}HQ.
}
Note that if we denote by $\overline{\Z/p}$
the co-constant $\underline{\Z/p}$-module (i.e. which is $\Z/p$ on both the fixed and free orbit, and
the corestriction is $1$ while the restriction is $0$), then we have
\beg{ecoconst}{\Sigma^{2-\beta}H\underline{\Z/p}=H\overline{\Z/p}.
}
Now note that at $p=2$, if we denote
the real sign representation by $\alpha$, we have
$$HQ=\Sigma^{1-\alpha}H\underline{\Z/2},$$
so
$$HQ\wedge_{H\underline{\Z/2}}HQ=\Sigma^{2-\beta}H\underline{\Z/2}=H\overline{\Z/2}.$$
For $p>2$, however, the situation is more complicated. Recall that for $1\leq i\leq p$, we have $\F_p[\Z/p]$-modules
$L_i$ on which the generator $\gamma$ of $\Z/p$ acts by a Jordan block of size $i$. For any
$\F_p[\Z/p]$-module, we have a corresponding
$\underline{\Z/p}$-module $\underline{V}$ where on the fixed orbit, we have $V^{\Z/p}$, and the 
restriction is inclusion (see \cite{sk, lewis}). Recall from \cite{sk} that the $\underline{\Z/p}$-modules
$\underline{L_p}$, $\underline{L_1}=\underline{\Z/p}$ are projective, so to calculate \rref{ecococo1},
we can use a projective resolution of $Q$ by these $\underline{\Z/p}$-modules. We have a short exact sequence
\beg{ecocos1}{0\r\underline{L_{p-1}}\r \underline{L_p}\r Q\r 0
}
and we also have a short exact sequence
\beg{ecocos2}{0\r \underline{L_{p+1-j}}\r\underline{L_1}\oplus\underline{L_p}\r\underline{L_j}\r 0.}
Splicing together short exact sequence \rref{ecocos1} with alternating
short exact sequences \rref{ecocos2} for $j=p-1$, $j=2$, we therefore obtain
a projective resolution of $Q$ of the form
\beg{ecocos3}{
\dots \underline{L_1}\oplus \underline{L_p}\r\dots\r \underline{L_1}\oplus \underline{L_p}\r \underline{L_p}
}
We also have 
$$\underline{L_p}\otimes_{\underline{\Z/p}} Q=\underline{L_p},$$
so tensoring \rref{ecocos3} with $Q$ over $\underline{\Z/p}$ gives a chain complex of $\underline{\Z/p}$-modules
of the form
\beg{ecocos4}{\dots Q\oplus \underline{L_p}\r\dots\r Q\oplus \underline{L_p}\r \underline{L_p},}
which gives the following

\begin{proposition}\label{pcocos}
For all primes $p$, we have:
$$Q\otimes_{\underline{\Z/p}}Q=\overline{\Z/p}.$$
For $p=2$, 
$$Tor_i^{\underline{\Z/2}}(Q,Q)=0\;\text{for $i>0$},$$
$$(HQ\wedge_{H\underline{\Z/2}}HQ)=H\overline{\Z/2}.$$
For $p>2$,
$$Tor_i^{\underline{\Z/p}}(Q,Q)=\Phi\;\text{for $i>0$},$$
and we have
\beg{eqqqq}{(HQ\wedge_{H\underline{\Z/p}}HQ)_\star=\Sigma^{2-\beta}HQ_\star\oplus 
\Sigma^2H\underline{\Z/p}_\star^\phi
}
as $H\underline{\Z/p}_\star$-modules.
\end{proposition}

\noindent
{\bf Remark:} It is important to note that the smash product studied in Proposition \ref{pcocos} is over
$H\underline{\Z/p}$. Over $H\underline{\Z}$, the answer is different, and in fact, the homotopy of the smash product is
finite-dimensional (see remark after Proposition \ref{p44} below).

\begin{proof}
All the statements except \rref{eqqqq} follow directly from evaluating the homology of \rref{ecocos4}. The reason the case of $p=2$ is special
is that then we have $2=p$, $p-1=1$, so there will be additional maps alternatingly canceling the fixed points for higher
$Tor$.

\vspace{3mm}
To prove \rref{eqqqq}, first note that its part concerning the $\Z$-graded line 
follows from the other statements. The $R$-graded
case, however, needs more attention, since we are no longer dealing with objects in the heart, so a priori
we merely have a spectral sequence whose $E^2$-term is given by the sum of the $R$-graded coefficients of the
$H\underline{\Z/p}$-modules with the Postnikov decomposition given by the Mackey $Tor$-groups:
\beg{emackeysss}{
E_{rs}^2=H(Tor^{\underline{\Z/p}}_s(Q,Q))_{r+*\beta}\Rightarrow (HQ\wedge_{H\underline{\Z/p}}HQ)_{r+s+*\beta}
}
(The indexing may seem reversed from what one would expect. Note, however, that it is induced by
a cell filtration on representation sphere spectra.)
We need to investigate this spectral sequence.

To this end,
it is worthwhile pointing out an alternative approach. Since all the $H\underline{\Z/p}$-modules considered above have very easy Borel homology, essentially the same information can be recovered by working
on geometric fixed points. Now for $H\underline{\Z/p}$-modules $A,B$, one has
\beg{ecocos5}{
A^\phi\wedge_{H\underline{\Z/p}^\phi}B^\phi=(A\wedge_{H\underline{\Z/p}}B)^\phi.
}
Also, we have a coherent equivalence of categories between $H\underline{\Z/p}^\phi$-modules and 
$(H\underline{\Z/p}^\phi)^{\Z/p}$-modules. Now we have
$$B:=H\underline{\Z/p}^\phi_*=\Z/p[t]\otimes\Lambda[u]$$
with $|t|=2$, $|u|=1$, while
$$J:=HQ^\phi_*=(u,t),$$
(by which we denote that ideal in $B$). By \cite{ekmm}, we therefore have a spectral sequence of the
form
\beg{ecocos10}{Tor^{B}_{r}(J,J)_s\Rightarrow (HQ\wedge_{H\underline{\Z/p}}HQ)^\phi_{r+s}.
}
To calculate the left hand side of \rref{ecocos10}, we have a $B$-resolution $C$ of $J$ of the form
\beg{ecocos11}{\diagram
\dots\rto & B[4]\rto^{-u}\drto^t & B[3]\rto^{-u}\drto^t&B[2]\\
\dot\rto& B[3]\rto_u & B[2]\rto_u& B[1].
\enddiagram
}
Tensoring over $B$ with $J$ and taking homology, we get
$$J\otimes_B J=\Z/p\{u\otimes u,u\otimes t, t\otimes u\}\oplus B\{t\otimes t\}$$
(where the braces indicate a sum of copies indexed by the
given elements) with the $B$-module structure the notation suggests, while
$$Tor_i^B(J,J)=\Z/p\{u,t\}[i+1]\;\text{for $i>0$}.$$
Thus, the spectral sequence \rref{ecocos10} is given by 
$$
E_{rs}^2=\left\{\begin{array}{ll}
\Z/p & \text{if $r=0$ and $s=2,4,5,6,\dots$}\\
\Z/p\oplus \Z/p & \text{if $r=0$ and $s=3$}\\
\Z/p & \text{if $r>0$ and $s=r+2,r+3$}\\
0 & \text{else.}
\end{array}
\right.
$$
Moreover, one can show that the spectral sequence \rref{ecocos10} collapses since
the only possible targets of differentials are the $B\{t\otimes t\}$ in filtration degree $0$, but by comparison with
$H\underline{\Z/p}$, we see that those elements cannot be $0$, since they inject into Borel homology
by the connecting map. Thus, we see that
\beg{emackeysss1}{
(HQ\wedge_{H\underline{\Z/p}}HQ)^\phi_n=\left\{
\begin{array}{ll}
\Z/p &\text{for $n=2$}\\
\Z/p\oplus \Z/p & \text{for $n=3,4,\dots$}\\
0 & \text{else.}
\end{array}
\right.
}
\begin{figure}
\includegraphics{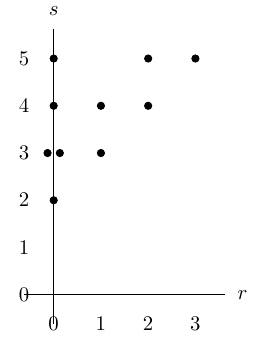}
\caption{The spectral sequence for $(HQ\wedge_{H\underline{\Z/p}}HQ)^\phi_*$}
\end{figure}
Now the $E2$-term \rref{emackeysss} on the $(?)^\phi$-level (obtained by inverting $b$)
is ``off by one" in the sense that we obtain
$$
\begin{array}{ll}
\Z/p &\text{for $n=1$}\\
\Z/p\oplus \Z/p & \text{for $n=2,3,4,\dots$}\\
0 & \text{else.}
\end{array}
$$
This, in fact, detects a single $d^2$-differential in the spectral sequence \rref{emackeysss} originating in the
$H\overline{\Z/p}_\star$-part in degrees 
$$2-n\beta,\;  n=1,2,3,\dots$$
and also proves there cannot be any other differentials, thus yielding \rref{eqqqq} (see Figures 5 and 6).
\end{proof}

\begin{figure}
\includegraphics{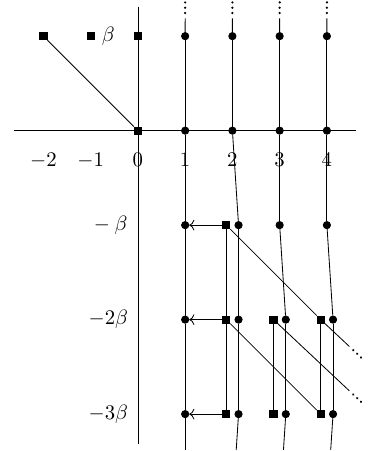}
\caption{The spectral sequence \rref{emackeysss}}
\end{figure}

\begin{figure}
\includegraphics{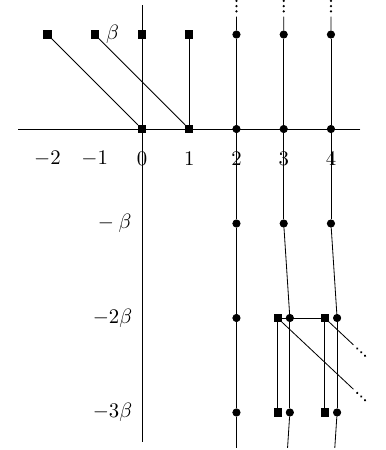}
\caption{$(HQ\wedge_{H\underline{\Z/p}}HQ)_\star$}
\end{figure}

\vspace{3mm}

\vspace{3mm}
Now let $T$ denote the $\Z/p$-equivariant suspension spectrum of the cofiber of the second
desuspension of the $\Z/p$-equivariant
based degree $p$ map
\beg{eswsw}{S^\beta\r S^2.}
(This spectrum was denoted by $T(\theta)$ in \cite{sw}.)
We also denote $HT=H\underline{\Z/p}\wedge T.$ We shall see in Section \ref{slens} that the spectrum
$T$ (up to suspension) maps into the based suspension spectrum of the $\Z/p$-equivariant
lens space $B_{\Z/p}(\Z/p)$, and therefore, the connecting map of the $H\underline{\Z/p}_\star$-homology
long exact sequence of \rref{eswsw} is an isomorphism on the $\Z/p$ in degree $\beta$ (which is the only
degree in which it can be non-trivial for dimensional reason). 

We also see that the
coefficients of $HT$ suggest the possibility of a filtration whose associated
graded pieces are wedges of suspensions of $HM$. Indeed, from the universal property, we readily
construct a morphism
$$HT\r HM,$$
which leads to a cofibration sequence of the form
\beg{eswsw1}{
\Sigma HM\r HT\r HM,
}
which splits additively on $R$-graded coefficients. The $\Z$-graded coefficients of $HT$,
which are $\Z/p$ in degrees $-1$ and $1$, and $\Z/p\oplus\Z/p$ in degree $0$,
generate the two $HM$-copies in \rref{eswsw1} in the above sense.

\begin{figure}
\includegraphics{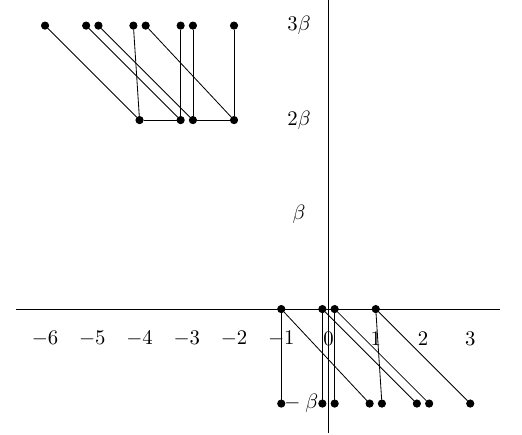}

\caption{The coefficients $HT_\star$}

\end{figure}

The cofibration \rref{eswsw1} can also be seen on the level of Mackey functors. Desuspending \rref{eswsw}
by $\beta$ and smashing with $H\underline{\Z/p}$, we get, by \rref{ecoconst},
a map of the form
$$H\underline{\Z/p}\r H\overline{\Z/p}.$$
Its cofiber can then be realized by a Mackey chain complex 
\beg{ecocos20}{
\underline{\Z/p}\r \overline{\Z/p}
}
set in homological degrees $0,1$, where the differential is $1$ on the fixed orbit and $0$ on the free orbit (this
is, essentially, the only non-trivial possibility). In the derived category of $\underline{\Z/p}$-modules, then, 
\rref{ecocos20} obviously maps to $Q$, with the kernel quasiisomorphic to $Q[1]$.

\vspace{3mm}

The cofibration \rref{eswsw1} does not split. To see that, we observe that 

\begin{proposition}\label{p44}
In the derived category of $p$-local $\Z/p$-equivariant spectra, we have an equivalence
\beg{eswsw2}{
T\wedge T\sim T\vee \Sigma^{\beta-1}T.
}
\end{proposition}

\begin{proof}
The strategy is to smash two copies of the cofibration
\rref{eswsw} together.
This gives a cofibration sequence
\beg{ettsplit}{T\r T\wedge T\r \Sigma^{\beta-1}T.
}
We need to show that the first map \rref{ettsplit} has a left inverse in the derived category. 
To this end, smashing the two cofibration sequences introduces two increasing filtrations on $T\wedge T$ where
the associated graded pieces in degree $0,1,2$ are
\beg{ettsplit1}{S^0, \;S^{\beta-1}\vee S^{\beta-1},\; S^{2\beta-2},}
respectively. We obtain an obvious splitting
$$F_1(T\wedge T)\r T$$
whose composition with the projection
$$T\r S^{\beta-1}$$
further also extends to $T\wedge T$. Therefore, the obstruction to constructing the splitting lies in
the $RO(\Z/p)$-graded $\Z/p$-equivariant homotopy group
\beg{ettsplit2}{\pi_{2\beta-3}^{\Z/p}S^0.
}
To study the group \rref{ettsplit2}, we consider the fibration
$$S^0\r S^{2\beta}\r\Sigma (S(2\beta)_+).$$
Since $\pi_0^{\Z/p}S^3=0$, we can represent a class in \rref{ettsplit2} by a $\Z/p$-equivariant stable map
\beg{ettsplit3}{\alpha:S(2\beta)_+\r S^2.}
The source of \rref{ettsplit3} is a free spectrum, and $\alpha$ is $0$ when restricted to the $1$-skeleton. 
On $S(2\beta)_2/S(2\beta)_1$, the group of homotopy classes of possible maps to $S^2$ is $\Z$. However,
the attaching map of the free $2$-cell to the free $1$-cell is $N_p=1+\gamma+\dots+\gamma^{p-1}$
and thus $p$ times the generator is homotopic to $0$. Thus, the group is actually $\Z/p$. All of these maps
extend to $S(2\beta)$ by the homotopy addition theorem, while the group of possible choices of the extension
lies in
$$\pi_1^{\{e\}}(S^0),$$
which has no $p$-primary component for $p>2$. 
Thus, away from $2$, the obstruction group is $\Z/p$ and is the same as on the level of $H\underline{\Z}$-modules.

In the category of $H\underline{\Z}$-modules, we observe that $\Sigma^{2-\beta}T$ is the homotopy cofiber
of a stable map
\beg{ettmack1}{S^0\r S^{2-\beta}.}
Smashing \rref{ettmack1} with $H\underline{\Z}$ can be realized as the map $\phi$ from the constant Mackey functor
$\underline{\Z}$ to the co-constant Mackey functor $\overline{\Z}$ which is $p$ on the free orbit and $1$ on the
fixed orbit:
\beg{ettmack2}{\phi:\underline{\Z}\r\overline{\Z}.}
We see that this map $\phi$ is injective and its cokernel is $Q$. Thus, we have:
\beg{ettmack3}{\Sigma^{2-\beta}H\underline{\Z}\wedge T=HQ.}
Now denoting by $\mathcal{L}_p$ the principal projective $\underline{\Z}$-module on the free orbit (i.e. the unique
$\underline{\Z}$-module which is the integral regular representation $\mathcal{L}_p$ on the free orbit and $\Z$ on
the fixed orbit, there is a $\underline{\Z}$-projective resolution of $\overline{\Z}$ of the form
\beg{ettmack4}{
\diagram \underline{\Z}\rto^\subset&\underline{\mathcal{L}_p}\rto^{1-\gamma}&\underline{\mathcal{L}_p}.
\enddiagram
}
Thus, given \rref{ettmack2}, we have a $\underline{\Z}$-resolution of $Q$ of the form
\beg{ettmack5}{
\diagram \underline{\Z}\rto^\subset&\underline{\Z}\oplus\underline{\mathcal{L}_p}\rto^{\subset\oplus(1-\gamma)}&\underline{\mathcal{L}_p}.
\enddiagram
}
Thus, $H\underline{\Z}\wedge\Sigma^{4-2\beta}T\wedge T$ can be realized by tensoring \rref{ettmack5} with
$Q$ over $\underline{\Z}$, which gives a chain complex of $\underline{\Z}$-modules of the form
\beg{ettmack6a}{
\diagram Q\rto^{0\oplus\subset}&Q\oplus\underline{L}_p\rto^{\subset\oplus(1-\gamma)}&\underline{L}_p
\enddiagram
}
with the last term in degree $0$, which can also be rewritten as the two-stage chain complex of 
$\underline{\Z}$-modules
\beg{ettmack6}{\diagram
Q\oplus\overline{L}_{p-1}\rto^{\subset\oplus(1-\gamma)}&\underline{L}_p
\enddiagram
}
in degrees $1,0$. Now we have a cofibration sequence
\beg{ettmack7}{\Sigma^{4-2\beta}T\r \Sigma^{4-2\beta}T\wedge T\r \Sigma^{3-\beta}T.
}
We see that the last term can be realized by the chain complex 
\beg{ettmack8}{Q\r 0
}
in degrees $0,1$. There exists a chain map $\psi$ from \rref{ettmack6} to \rref{ettmack8} which is identity on $Q$ and
$0$ on the other components. This map, in fact, has a right inverse given by
$$(1,-(1-\gamma)^{p-2}).$$
Moreover, $Ker(\psi)$ is the chain complex in degrees $1,0$ of the form
\beg{ettmack9}{\diagram
\overline{L}_{p-1}\rto^{1-\gamma}&\underline{L}_p
\enddiagram
}
which can be also written as
\beg{ettmack10}{\diagram
Q\rto^\subset&\underline{L}_{p}\rto^{1-\gamma}&\underline{L}_p
\enddiagram
}
which is \rref{ettmack4} tensored over $\underline{\Z}$ with $Q$, and thus represents
$$H\underline{\Z}\wedge S^{2-\beta}\wedge \Sigma^{2-\beta}T.$$

Thus, we have proved our splitting after smashing with $H\underline{\Z}$ 
in the category of $H\underline{\Z}$-modules. To prove the statement spectrally, however, we need to be 
even more precise, since the surjection from \rref{ettmack6} to \rref{ettmack8} is not unique: It could also 
be non-zero on the second summand (and those surjections do not split). We need to prove specifically that
the surjection from \rref{ettmack6} to \rref{ettmack8} induced by the second map \rref{ettmack7} is $0$
on the second component of the source of \rref{ettmack6}. However, to this end, it suffices to prove that 
the map $\psi$ vanishes when composed with the canonical chain map from \rref{ettmack4} to \rref{ettmack6a}.
We see however that this map is obtained by smashing with $H\underline{\Z}$ the composition
$$S^{2-\beta}\r \Sigma^{2-\beta}T\r\Sigma^{4-2\beta}T\wedge T\r \Sigma^{3-\beta}T,$$
which is $0$ since it involves composing two consecutive maps in a cofibration sequence. Thus, 
our statement is proved.
\end{proof}

\vspace{3mm}

\noindent
{\bf Remark:} From the above proof, one can also read off the homotopy groups of $HQ\wedge_{H\underline{\Z}}HQ$,
which are concentrated in finitely many degrees
(compare with Proposition \ref{pcocos}).

Smashing \rref{eswsw2} with $H\underline{\Z/p}$, we see that 
$$HT\wedge_{H\underline{\Z/p}}HT=HT\vee \Sigma^{\beta-1}HT,$$ 
which additively splits 
as a direct sum of $HM_\star$ suspended by $0,1,\beta,\beta-1$.This is not what would happen if
the cofibration of $H\underline{\Z/p}$-modules \rref{eswsw1} split: the higher derived
terms would appear. This will play a role of our description of the 
$\Z/p$-equivariant Steenrod algebra in the subsequent sections.

\begin{figure}
\includegraphics{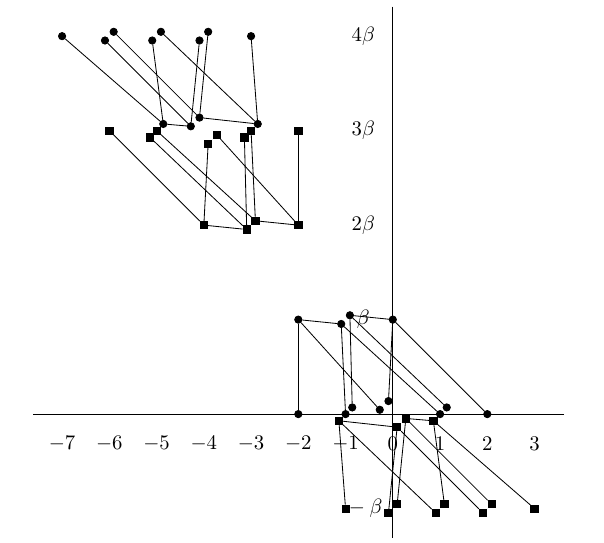}
\caption{The coefficients $(HT\wedge_{H\underline{\Z/p}}HT)_\star$}

\end{figure}

One may, in fact, ask how the ``tidy" behavior \rref{eswsw2} is even possible on $H\underline{\Z/p}$-homology,
given the infinitely many higher $Tor$'s of $Q$ with itself, computed in Proposition \ref{pcocos}. 
We present an explanation in terms of geometric fixed points: The resolution \rref{ecocos11}
in fact gives a short exact sequence of $B$-modules of the form
$$0\r J[1]\r B[1]\oplus B[2]\r J\r 0.$$
On the level of geometric fixed points, the cofibration \rref{eswsw1} in fact realizes this extension.

\vspace{3mm}
\noindent
{\bf Comment:}
It is worth noting that each of the traceless indecomposable modular representations $L_i$,
$i=1,\dots,p-1$
gives rise to a Mackey functor (which we also denote by $\widetilde{L_i}$) equal to $L_i$ on the free orbit
and $0$ on the fixed orbit. Thus, $\widetilde{L_1}=Q$.
The $H\underline{\Z/p}$-modules $H\widetilde{L_i}$, $i=1,\dots,p-1$ all have
additively isomorphic $RO(\Z/p)$-graded coefficients, even though no morphism of spectra
induces this isomorphism (passing to Borel homology, this says there is no map between
$L_i$, $L_j$ for $i\neq j$ which would induce an isomorphism in group homology).
Of course, the non-equivariant coefficients of $HL_i$ are $L_i$ in degree $0$, so we see they all
are of
different dimensions. The reason we only encounter $H\widetilde{L_1}=HQ$ in our calculations is that
in the Borel homology spectrum of $H\underline{\Z/p}\wedge H\underline{\Z/p}$ is a wedge
sum of suspended copies of $H\underline{\Z/p}^b$.

This raises the question as to whether in general a morphism of spectra $f:X\r Y$ which induces
an isomorphism in $RO(\Z/p)$-graded coefficients is a weak equivalence. This is obviously true for $p=2$ 
by the cofibration sequence
$$\Z/2_+\r S^0\r S^\alpha,$$
but it is false for $p>2$: Consider the Mackey functor which is equal to the traceless complex
representation $\beta$ on the free orbit and $0$ on the fixed orbit. (We will also denote it by $\beta$.)
Then the $\Z/p$-Borel homology spectrum $H\beta^b$ has trivial $RO(\Z/p)$-graded coefficients, since
the cohomology theory is $(\beta-2)$-periodic, and the group homology of $\Z/p$ with coefficients in $\beta$
is $0$. (Also note that for $p=2$, this Borel homology will be non-zero in degree $\alpha-1$
where $\alpha$ is the $1$-dimensional real sign representation.)

On the other hand, call an equivariant spectrum $X$ {\em $p$-complete} when its canonical map
into the homotopy inverse limit of $X\wedge M\Z/p^r$ is an equivalence.
Then a morphism $f:X\r Y$ of $p$-complete bounded below $\Z/p$-equivariant spectra which 
induces an isomorphism of $RO(\Z/p)$-graded coefficients is a weak equivalence. To see this, since we already
know $f^\phi$ is an equivalence, it suffices to consider the case when $X,Y$ are free. Equivalently,
we must show that a bounded below free $\Z/p$-spectrum whose fixed point coefficients are
$0$ is $0$. So assume it is not $0$, and consider the bottom dimensional degree non-equivariant 
$\Z[\Z/p]$-module $V$ of its coefficients. But then $V/(1-\gamma)\neq 0$, since on a modular representation of
$\Z/p$, $1-\gamma$ is never onto. Thus, the coefficients of the fixed point spectrum are non-zero in
the same degree.

We do not know whether the bounded below assumptions can be removed.

\section{Cohomology of the equivariant projective spaces and lens spaces}\label{slens}

The $\Z/p$-equivariant complex projective space $\C P^\infty_{\Z/p}$ can be identified
with the space of complex lines on the complete complex $\Z/p$-universe $\mathcal{U}$.
An explicit decomposition 
\beg{eflag}{\mathcal{U}=\bigoplus_{i\in \N_0} \alpha_i}
is called a {\em flag}. It leads to a filtration
$$F_n(\C P^\infty)=P(\alpha_0\oplus\dots\oplus \alpha_n)$$
We have
$$F_n(\C P^\infty)/F_{n-1}(\C P^\infty)\cong S^{\alpha_n^{-1}(\alpha_0\oplus\dots\oplus \alpha_{n-1})}.$$
This leads to a spectral sequence
\beg{ess1}{E_1=\bigoplus_{n\in\N_0}H\underline{\Z/p}_{\star-\alpha_n^{-1}(\alpha_0\oplus\dots\oplus 
\alpha_{n-1})}
\Rightarrow H\underline{\Z/p}^\star\C P^\infty_{\Z/p}.}

\begin{figure}

\includegraphics{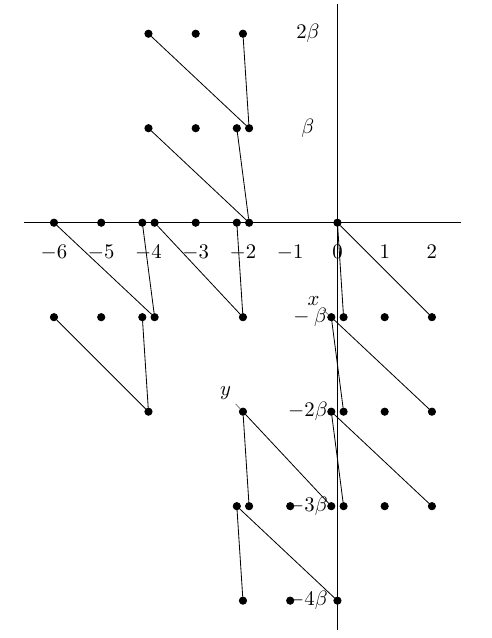}

\caption{The flag spectral sequence for $B_{\Z/3}(S^1)$}

\end{figure}

Whether or not this collapses depends on the flag. For the {\em regular flag}
$$\alpha_i=\beta^i,$$
(remembering our convention on indexing), the free generators of the copies of $H\underline{\Z/p}_\star$
are in dimensions
\beg{ess2}{\begin{array}{l}0,-\beta,-2\beta,\dots,-(p-1)\beta,\\
-(p-1)\beta-2,-p\beta-2,\dots,-2(p-1)\beta-2,\\
-2(p-1)\beta-4,\dots,-3(p-1)\beta-4,\\
\dots
\end{array}
}
We see from Proposition \ref{p1} that there is no element in \rref{ess1} in total dimension $-\beta-1$
or $-3-(p-1)\beta$, and therefore 
the generator $x$ in dimension $-\beta$ and the generator $y$ in dimension $-2-(p-1)\beta$
are permanent cycles, and thus, the spectral sequence collapses, and thus 
$H\underline{\Z/p}^\star\C P^\infty_{\Z/p}$ is a free $H\underline{\Z/p}_\star$-module.

\begin{proposition}\label{p33}
One has
$$H\underline{\Z/p}^\star\C P^\infty_{\Z/p} =H\underline{\Z/p}_\star[x,y]/(\sigma^{-2}y-x^p+b^{p-1}x).
$$
\end{proposition}

\begin{proof}
The collapse of the spectral sequence was already proved. Thus, what remains to prove is the multiplicative
relation. To this end, we work in Borel cohomology. (This could, in principle, generate counterterms in the
$\Gamma^\prime$ tail, but since it is $\sigma^{-2}$-divisible, that could be corrected by a different choice
of the generator $y$.) 

Now in Borel cohomology, we are essentially working in the cohomology of the space $\C P^\infty\times B\Z/p$.
From this point of view, it is convenient to treat the periodicity $\sigma^{2}$ as the identity, so from this
point of view, we have
$$H^*(\C P^\infty\times B\Z/p;\Z/p)=\Z/p[x,b]\otimes \Lambda[u]$$
where $x,b$ have cohomological dimension $2$ and $u$ has cohomological dimension $1$.
The computation of the element $y$ from this point of view then amounts to computing the Euler class
of the regular complex representation of $\Z/p$. This is
$$x(x+b)(x+2b)\dots(x+(p-1)b)=x^p-b^{p-1}x.$$
\end{proof}

Now similarly as in Milnor \cite{milnor}, we have, for a $\Z/p$-space $X$, a multiplicative map
\beg{emilnor1}{\lambda:H\underline{\Z/p}^\star(X)=F(X,H\underline{\Z/p})_{-\star}
\r F(X,H\underline{\Z/p}\wedge H\underline{\Z/p})_{-\star}.}
When $H\underline{\Z/p}^\star X$ is a flat $H\underline{\Z/p}_\star$-module (as is the case for $X=\C P^\infty_{\Z/p}$,
this can be written as
\beg{emilnor1111}{\lambda:H\underline{\Z/p}^\star(X)\r H\underline{\Z/p}^\star(X)\widehat{\otimes} A_\star}
where the $\widehat{\otimes}$ denotes the tensor product completed at $(b)$. For $X=\C P^\infty_{\Z/p}$,
we get
\beg{emilnor2}{\lambda(x)=x\otimes 1+\sum_{n\geq 1}y^{p^{n-1}}\otimes\underline{\xi}_n
}
and
\beg{emilnor3}{\lambda(y)=y\otimes 1+\sum_{n\geq 1}y^{p^n}\otimes\underline{\theta}_n
}
where the dimensions are given by
$$|\underline{\xi}_n|=2p^{n-1}+(p^n-p^{n-1}-1)\beta,$$
$$|\underline{\theta}_n|=2(p^{n}-1)+(p-1)(p^n-1)\beta.$$
From co-associativity, we can further conclude that, writing
$$\widetilde{\psi}(t)=\psi(t)-t\otimes1-1\otimes t, $$
we have
\beg{ecoprod1}{\widetilde{\psi}(\underline{\xi}_n)=\sum\underline{\theta}_i^{p^{n-i-1}}\otimes\underline{\xi}_{n-i},
}
\beg{ecoprod2}{\widetilde{\psi}(\underline{\theta}_n)=\sum\underline{\theta}_i^{p^{n-i}}
\otimes \underline{\theta}_{n-i}
.}

\vspace{3mm}
The picture becomes a little less tidy when we calculate $H\underline{\Z/p}^\star B_{\Z/p}(\Z/p)$.
For a model of $B_{\Z/p}(\Z/p)$, we use the quotient of the unit sphere in $\mathcal{U}$ by the action 
of $\Z/p\subset S^1$. Now a flag \rref{eflag}, we have a filtration with
$$F_{2n+1}B_{\Z/p}(\Z/p)=S(\alpha_0\oplus\dots\oplus\alpha_n)/(\Z/p)$$
$$F_{2n}B_{\Z/p}(\Z/p)=\{(x_0,\dots,x_n)\in S(\alpha_0\oplus\dots\oplus\alpha_n)\mid Arg(x_n)=2k\pi/p\}/
(\Z/p).$$
We have
$$F_{2n}/F_{2n-1}\cong S^{(\alpha_0\oplus\dots \oplus\alpha_{n-1})\alpha_n^{-1}},$$
$$F_{2n+1}/F_{2n}\cong S^{(\alpha_0\oplus\dots \oplus\alpha_{n-1})\alpha_n^{-1}\oplus 1_\R}.$$
Thus, we have a spectral sequence
\beg{esss1a}{E_1=\bigoplus_{n\in \N_0}H\underline{\Z/p}_\star F_n/F_{n-1}\Rightarrow 
H\underline{\Z/p}_\star B_{\Z/p}(\Z/p)
}
If we use the regular flag, the generators of the summands will be in (homological) degrees:
$$\begin{array}{l}
0,-1,-\beta, -\beta-1,\dots, -(p-1)\beta, -(p-1)\beta-1,\\
-(p-1)\beta-2, -(p-1)\beta-3,\dots,-2(p-1)\beta-2, -2(p-1)\beta-3,\\
-2(p-1)\beta-4,\dots
\end{array}
$$

\begin{figure}
\includegraphics{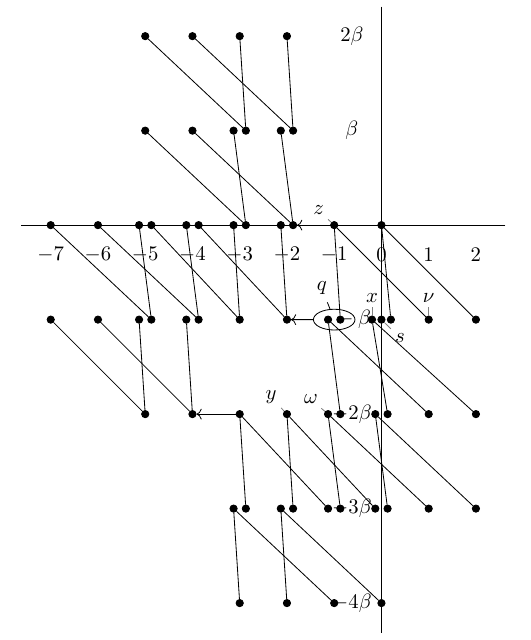}
\caption{The flag spectral sequence for $B_{\Z/3}(\Z/3)$}
\end{figure}

The difference however now is that in the spectral sequence \rref{esss1a}, the generator $z$ in
degree $-1$ supports a $d_1$ differential. To see this, otherwise, it would be a permanent cycle,
and hence so would its Bockstein.  Now the image of $z$ in Borel cohomology is a non-trivial element of
$$H^1(\Z/p\times\Z/p;\Z/p),$$
so the image of $\beta(z)$ in Borel cohomology would be non-zero. However, we see that
in the $E_1$-term of the spectral sequence \rref{esss1a}, all the elements of dimension 
$-2$ have image $0$ in Borel cohomology. 

There is a unique target of this differential, and all other differentials originate in
$$z\cdot H\underline{\Z/p}^\star \C P^\infty_{\Z/p}.$$
One also notes that 
\beg{eeee1}{z\cdot x^{p-1}}
is a permanent cycle.
Bookkeeping leads to the following:

\vspace{5mm}

\begin{proposition}\label{p10}
We have
\beg{ep101}{\begin{array}{l}H\underline{\Z/p}^\star B_{\Z/p}(\Z/p)=
\\
H\underline{\Z/p}_\star[y]\otimes
\Lambda[z\cdot x^{p-1}]\oplus 
(HM_\star[x,y]/(\sigma^{-2}y-x^p+b^{p-1}x))
\otimes \Z/p\{x,\nu\}.
\end{array}}
(See Figure 10 for the element $\nu$.)
\end{proposition}
\qed

\begin{figure}


\includegraphics{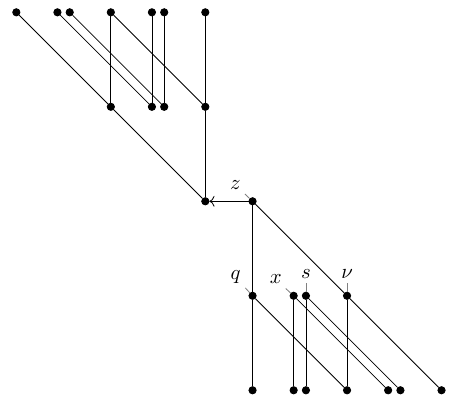}

\caption{The $x$ and $z$ towers in the $B_{\Z/3}(\Z/3)$-flag spectral sequence, p.1}

\end{figure}

\begin{figure}


\includegraphics{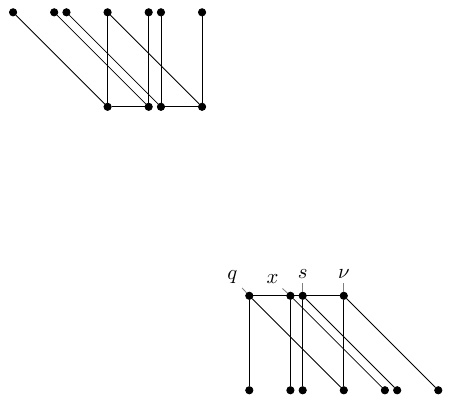}

  \caption{The $x$ and $z$ towers in the $B_{\Z/3}(\Z/3)$-flag spectral sequence, p.2}

\end{figure}

In Proposition \ref{p10}, the ``polynomial generators" at $HM_\star$ just mean
suspension by dimensional degree of the given monomial. 
In particular, we have canonical elements $q\in H^{1+\beta}B_{\Z/p}(\Z/p)$, $s\in H^{\beta}
B_{\Z/p}(\Z/p)$,  $\nu\in H^{\beta-1}
B_{\Z/p}(\Z/p)$ represented 
by 
\beg{ep102}{bz-xu\in x\cdot HM_\star, \;z\sigma^{-2}u\in \nu\cdot HM_\star, z\sigma^{-2}\in \nu\cdot HM_\star.}
(See figure 10; whiskers point to the names of the elements concerned.)
In \rref{ep102}, the name of the elements comes from Borel cohomology, which we write as 
elements of the appropriate terms of \rref{ep101}.

We also have an element $\omega\in H^{(p-1)\beta+1}(B_{\Z/p}(\Z/p))$ which represents
\rref{eeee1}. One notes, however, that this is only correct in the associated graded object of
our filtration. To determine the exact image in Borel cohomology, we note that we must have
$$\beta(\omega)=y,$$
since $y$ is the additive generator of $H\underline{\Z/p}^{(p-1)\beta+2}B_{\Z/p}(\Z/p)$. This gives
\beg{eomegaomega}{\omega\mapsto z(\sigma^{2-2p}t^{p-1}-b^{p-1}).}

Now in \rref{emilnor1}, we can write (using Proposition \ref{p44}):
$$\lambda(q)=q\otimes 1 +\sum_{n\geq 1} y^{p^{n-1}}\otimes \widehat{\xi}_{n},$$
\beg{essss}{\lambda(s) = s \otimes 1 + q \otimes (\tau_0 + b^{p-2} \widehat{\xi}_1) +
\sum_{n \geq 1} y^{p^{n-1}} \otimes \widehat{\tau}_n,}
\beg{enununu}{\lambda(\nu) = \nu \otimes 1 - q \otimes b^{p-2} \underline{\xi}_1 + x \otimes
(\tau_0 + b^{p-2} \widehat{\xi}_1) + \sum_{n\geq 1} y^{p^{n-1}} \otimes \underline{\tau}_n,}
thus defining the following new elements of $A_\star$:
$$\begin{array}{ll}
\widehat{\xi}_n, & n\geq 1,\\
\tau_0,&\\
\widehat{\tau}_n, & n\geq 1,\\
\underline{\tau}_n, & n\geq 1.
\end{array}
$$
The somewhat complicated form of
the right hand side of formulas \rref{essss}, \rref{enununu} is forced by
considering which elements exist in 
$$H\underline{\Z/p}^\star B_{\Z/p}(\Z/p).$$ 
Specifically, the coefficient of $q$ in
\rref{essss} and of $x$ in \rref{enununu}, which have two terms, can be justified by comparison with geometric fixed points. Similarly
for the 
coefficient of $q$ in \rref{enununu}. Otherwise, the generators are being defined by their appearance on the right-hand side.

(Note: Existence of these elements in $A_\star$ can also be established directly by looking at the geometric fixed points
and studying divisibility by $b$ in Borel cohomology, which can be read off the calculations treated
in Section \ref{sborel} below.)

\vspace{3mm}

We can also write
$$\lambda(\omega)=\omega\otimes 1 +y\otimes \tau_0+\sum_{n\geq1} y^{p^n}\otimes\underline{\mu}_n +\dots,$$
where $\dots$ indicate that there will be summands attached to other terms in the decomposition \rref{ep101}.
We can, nevertheless, define $\underline{\mu}_n$ as the coefficient of $y^{p^n}$.
The dimensions of the elements defined above are then given by
$$|\tau_0|=1,$$
$$|\widehat{\xi}_n|=|\underline{\xi}_n|-1,$$
$$|\underline{\tau}_n|=|\underline{\xi}_n|+1,$$
$$|\widehat{\tau}_n|=|\underline{\tau}_n|-1=|\underline{\xi}_n|.$$
$$|\underline{\mu}_n|=|\underline{\theta}_n|+1.$$
(Note that one could make the convention $\underline{\theta}_0=1$, $\underline{\mu}_0=\tau_0$.)

Co-associativity then implies
\beg{ecoprod1a}{\widetilde{\psi}(\widehat{\xi}_n)=\sum\underline{\theta}_i^{p^{n-i-1}}\otimes\widehat{\xi}_{n-i},
}
\beg{ecoprod2a}{\widetilde{\psi}(\widehat{\tau}_n)=\sum\underline{\theta}_i^{p^{n-i-1}}
\otimes\widehat{\tau}_{n-i}+\widehat{\xi}_n\otimes
(\tau_0+\widehat{\xi}_1b^{p-2})
,}
\beg{ecoprod3a}{\widetilde{\psi}(\underline{\tau}_n)=\sum\underline{\theta}_i^{p^{n-i-1}}
\otimes\underline{\tau}_{n-i}+\widehat{\xi}_n\otimes\underline{\xi}_1b^{p-2}+\underline{\xi}_n\otimes
(\tau_0+\widehat{\xi}_1b^{p-2})
.}

We can now re-state Proposition \ref{p10} more precisely, in fact determining the multiplicative
structure of $H\underline{\Z/p}^\star B_{\Z/p}(\Z/p)$ completely:

\begin{proposition}\label{pzpzp}
The module $H\underline{\Z/p}^\star B_{\Z/p}(\Z/p)$ is additively isomorphic to a direct sum of
$$H\underline{\Z/p}_\star[y]\otimes\Lambda[\omega]$$
and (suspended) copies of $HM_\star$ generated on monomials of the form
$$y^nx^{i-1}\pi$$
where $n\in \N_0$, $i=1,\dots, p-1$, and $\pi$ stands for one of the symbols $x,\nu$. The multiplicative
structure is entirely determined by the canonical inclusion of the good tail into
$$H\underline{\Z/p}^\star\C P^\infty_{\Z/p}\otimes \Lambda[z,u].$$ 
In particular, the good tail is the quotient of
$$H\underline{\Z/p}_\star[x,y]\otimes\Lambda[s,q,\nu,\omega]$$
modulo the relations
$$x^p-b^{p-1}x=y\sigma^{-2},$$
$$\sigma^{-2}q=b\nu-(\sigma^{-2}u)x,$$
$$\sigma^{-2}s=(\sigma^{-2}u)\nu,$$
$$(\sigma^{-2} u)s=0,$$
$$(\sigma^{-2}u)q=bs,$$
$$qs=0,$$
$$q\nu=xs,$$
$$\nu s=0,$$
$$\omega s=\omega\nu=0,$$
$$\omega q=-sy,$$
$$\omega x=y\nu.$$
Additionally, we have
$$H\underline{\Z/p}^\star B_{\Z/p}(\Z/p)^d=H^2_{(b,\sigma^{-2})}(H\underline{\Z/p}_\star^g,
H\underline{\Z/p}^\star B_{\Z/p}(\Z/p)^g)[-1].$$
\end{proposition}
\qed

\vspace{3mm}

\section{Images in Borel cohomology}\label{sborel}

Similarly as in the $p=2$ case, it is convenient to consider the {\em Borel cohomology dual
Steenrod algebra}
$$A^{cc}_\star=F(E\Z/p_+,H\underline{\Z/p}\wedge H\underline{\Z/p})_\star.$$
We have
$$A^{cc}_\star=(A_*[\sigma^{2},\sigma^{-2}][b]\otimes \Lambda[u])^\wedge_{(b)}.$$
In $A^{cc}_\star$, which makes a formal Hopf algebroid, the non-equivariant coproduct relations hold, and we have
\beg{eru}{\rho^2:=\eta_R(\sigma^2)=\sigma^2+\sigma^{2p} b^{p-1} \xi_1+\dots+\sigma^{2p^n}
b^{p^n-1}\xi_n+\dots
}
\beg{erou}{\overline{u}:=\eta_R(u)=
u+\sigma^2 b\tau_0+\sigma^{2p}b^p\tau_1+\dots+\sigma^{2p^n}b^{p^n}\tau_n+\dots
}

\begin{proposition}\label{pppelements}
In $A_\star^{cc}$, we have
\beg{ero10}{
\underline{\xi}_n=\rho^{-2}\sigma^{2p^n-2p^{n-1}}\xi_n+\dots+\rho^{-2}\sigma^{2p^N-2p^{n-1}}b^{p^N
-p^n}\xi_N+\dots,
}
\beg{erou101}{\begin{array}{l}
\widehat{\xi}_n=\sigma^{2p^n-2p^{n-1}}(b\tau_n-\overline{u}\rho^{-2}\xi_n)+\dots\\
+\sigma^{2p^N-2p^{n-1}}(b\tau_N-\overline{u}\rho^{-2}\xi_N)b^{p^N-p^n}+\dots,
\end{array}
}
\beg{erou102}{
\underline{\tau}_n=
\rho^{-2}\sigma^{2p^n-2p^{n-1}}\tau_n+\dots+\rho^{-2}\sigma^{2p^N-2p^{n-1}}b^{p^N
-p^n}\tau_N+\dots \;,
}
\beg{erou103}{
\widehat{\tau}_n=\underline{\tau}_n\overline{u}=\rho^{-2}\sigma^{2p^n-2p^{n-1}}\tau_n\overline{u}+
\dots +\rho^{-2}\sigma^{2p^N-2p^{n-1}}\tau_N\overline{u}b^{p^N-p^n}+\dots\; .
}
\end{proposition}

\begin{proof}
Multiplying \rref{eru} by $\sigma^{-2}\rho^{-2}$, we get
\beg{ero1}{\sigma^{-2}=\rho^{-2}+\rho^{-2}\sigma^{2p-2}b^{p-1}\xi_1+\dots \rho^{-2}\sigma^{2p^n-2}b^{p^n-1}
\xi_n+\dots}
Now $\C P^\infty_{\Z/p}$ has the same $\Z/p$-Borel cohomology as $\C P^\infty$, 
which is
\beg{eborelcp}{H\underline{\Z/p}^\star_c\C P^\infty=\Z/p[t][\sigma^{2},\sigma^{-2}][b]\otimes\Lambda[u].
}
Further, in Borel cohomology, we can write
\beg{ecrosss}{x=\sigma^{-2}t}
and thus,
$$y=\sigma^{2-2p}t^p-tb^{p-1}.$$
Combining this with \rref{emilnor2}, we have
\begin{align*}
\lambda(t) & = \lambda(\sigma^2 x) = x \otimes \rho^2 + \sum_{n \geq 1} y^{p^{n-1}}
            \otimes \underline{\xi}_n \rho^2\\
  & = \sigma^{-2}t \otimes \rho^2 +  \sum_{n \geq 1} (\sigma^{2p^{n-1}-2p^n}t^{p^n} - b^{p^{n}-p^{n-1}}t^{p^{n-1}})
    \otimes \underline{\xi}_n \rho^2\\
  & = t \otimes (\sigma^{-2}-b^{p-1}\underline{\xi}_1) \rho^2 +  \sum_{n
    \geq 1} t^{p^n} \otimes (\sigma^{2p^{n-1}-2p^n} \underline{\xi}_n -
    b^{p^{n+1}-p^{n}} \underline{\xi}_{n+1})\rho^2
\end{align*}
Comparing with the non-equivariant result
\begin{equation*}
\lambda(t) = t \otimes 1 + \sum_{n \geq 1} t^{p^n} \otimes \xi_n,
\end{equation*}
we have
\begin{equation*}
\begin{cases}
  1 & =(\sigma^{-2}-b^{p-1}\underline{\xi}_1) \rho^2 \\
  \xi_n & = (\sigma^{2p^{n-1}-2p^n} \underline{\xi}_n -
    b^{p^{n+1}-p^{n}} \underline{\xi}_{n+1})\rho^2,
\end{cases}
\end{equation*}
and so
\begin{equation*}
\begin{cases}
  \underline{\xi}_1  & =(\sigma^{-2}- \rho^{-2}) b^{1-p} \\
  \underline{\xi}_{n+1}  & = (\sigma^{2p^{n-1}-2p^n} \underline{\xi}_n -
    \rho^{-2}\xi_{n})b^{p^n-p^{n+1}}. \\
\end{cases}
\end{equation*}
(We work in Tate cohomology, into which the Borel cohomology embeds.)
Using \rref{ero1} and induction, we can prove \rref{ero10}.

Now apply $\lambda$ to $\sigma^{-2}y = x^p - b^{p-1}x$, we have
\begin{align*}
& y \otimes \rho^{-2} + \sum_{n\geq 1} y^{p^n} \otimes \underline{\theta}_n
  \rho^{-2} \\
  & =
(x^p \otimes 1 + \sum_{n \geq 1} y^{p^n} \otimes \underline{\xi}_n)-
(b^{p-1}x \otimes 1 + \sum_{n \geq 1} y^{p^{n-1}} \otimes
            b^{p-1}\underline{\xi}_n) \\
  & =
(\sigma^{-2}y \otimes 1 - y \otimes  b^{p-1}\underline{\xi}_1) + \sum_{n \geq 1}
    y^{p^n} \otimes (\underline{\xi}_n - b^{p-1}\underline{\xi}_{n+1})
\end{align*}
Comparing coefficients, we get 
\beg{ero11}{\underline{\theta}_n\rho^{-2}=\underline{\xi}_n^p-\underline{\xi}_{n+1}b^{p-1}.
}
Note, in fact, 
that the relation \rref{ero11} is true on the nose (meaning not just on the image in Borel cohomology, i.e. modulo the 
derived tail) by Proposition \ref{p33}, and the Hopf pluri-algebroid relation between the product 
and the coproduct. 
One should also point out that, in particular,
\beg{ero12}{\rho^{-2}=\sigma^{-2}-\underline{\xi}_1b^{p-1}.
}
Now multiplying \rref{ero1} by $\overline{u}$, we get
$$\overline{u}
\sigma^{-2}=\overline{u}\rho^{-2}+\overline{u}\rho^{-2}\sigma^{2p-2}b^{p-1}\xi_1+\dots 
\overline{u}\rho^{-2}\sigma^{2p^n-2}b^{p^n-1}
\xi_n+\dots$$
Plugging in \rref{erou}, we get
\beg{erou100}{\begin{array}{l}
u\sigma^{-2}-\overline{u}\rho^{-2}+b\tau_0=\\
b^{p-1}\sigma^{2p-2}(\overline{u}\rho^{-2}\xi_1-b\tau_1)+
\dots + b^{p^n-1}\sigma^{2p^n-2}(\overline{u}\rho^{-2}\xi_n-b\tau_n)+\dots
\end{array}
}
In the Borel cohomology, apply $\lambda$ to $q = bz - xu = bz - \sigma^{-2}tu$, we get 
\begin{equation*}
\lambda(q) = \lambda(z)(1 \otimes b) - \lambda(t) (1 \otimes \rho^{-2} \bar{u}).
\end{equation*}
Plugging in the formula for $\lambda(q)$, $\lambda(z)$, $\lambda(t)$, we get
\begin{align*}
&(bz - \sigma^{-2}tu) \otimes 1 + \sum_{n \geq 0} (\sigma^{2p^{n}-2p^{n+1}}
  t^{p^{n+1}} - b^{p^{n+1}-p^{n}} t^{p^n}) \otimes \widehat{\xi}_{n+1} \\
& = (z \otimes b + \sum_{n \geq 0} t^{p^n} \otimes \tau_n b) - (t \otimes
  \rho^{-2} \bar{u} + \sum_{n \geq 1} t^{p^n} \otimes \xi_n \rho^{-2} \bar{u}).
\end{align*}
The coefficients of $z$ match on both sides.
Comparing coefficients of $t^{p^n}$, we get
\begin{equation*}
  \begin{cases}
  -\sigma^{-2}u  -b^{p-1} \widehat{\xi}_1 + \rho^{-2} \overline{u} - b \tau_0 & = 0\\
 \sigma^{2p^{n-1}-2p^{n}} \widehat{\xi}_{n} -  b^{p^{n+1}-p^{n}} \widehat{\xi}_{n+1}
  - \tau_n b + \xi_n \rho^{-2}
  \overline{u} &= 0,             
\end{cases}
\end{equation*}
and so
\begin{equation*}
  \begin{cases}
  \widehat{\xi}_1 & =(-\sigma^{-2}u + \rho^{-2} \overline{u} - b \tau_0) b^{1-p} \\
  \widehat{\xi}_{n+1} & = (\sigma^{2p^{n-1}-2p^{n}} \widehat{\xi}_{n} - \tau_n b + \xi_n \rho^{-2}
  \bar{u}) b^{p^{n}-p^{n+1}}.             
\end{cases}
\end{equation*}
Using \rref{erou100} and induction, we can prove \rref{erou101}.

Recall that we have in Borel cohomology
\begin{align*}
  q & = bz - \sigma^{-2}tu \\
  s & = \sigma^{-2}zu \\
  \nu & = \sigma^{-2}z
\end{align*}

From above, we have \rref{ero12} and also
\beg{eq:xi1hat}{
  b^{p-1} \widehat{\xi}_1  = -\sigma^{-2}u + \rho^{-2}\overline{u} - b\tau_0.}
Multiplying \rref{ero12} by $u$ (resp. $\overline{u}$) and adding to
\rref{eq:xi1hat}, we get
\begin{align}
  \label{eq:useful1}
   ub^{p-1}\underline{\xi}_1 + b^{p-1} \widehat{\xi}_1 + b\tau_0 & = \rho^{-2}(\overline{u}-u),
  \\
  \label{eq:useful2}
  \overline{u}b^{p-1}\underline{\xi}_1 + b^{p-1} \widehat{\xi}_1 + b\tau_0 & = \sigma^{-2}(\overline{u} - u).
\end{align}

Plugging in the Borel cohomology expression into formula \rref{enununu}
and comparing coefficients with 
\begin{equation*}
  \lambda(\nu) = \lambda(z)(1 \otimes \rho^{-2}) = z \otimes \rho^{-2} +
  \sum_{n \geq 0} t^{p^n}\otimes \tau_n\rho^{-2},
\end{equation*}
we must have that 
\begin{align}
\label{eq:nv1}
  \rho^{-2} & = \sigma^{-2} - b^{p-1} \underline{\xi}_1
  & \text{ coefficient of $z$ }\\
\label{eq:nv2}
  \rho^{-2} \tau_0& =  \sigma^{-2}(ub^{p-2} \underline{\xi}_1 + b^{p-2}
                    \widehat{\xi}_1 + \tau_0) - b^{p-1} \underline{\tau}_1
  &\text{ coefficient of $t$ }\\
  \label{eq:nv3}
  \rho^{-2} \tau_{n} & = \sigma^{2p^{n-1}-2p^{n}} \underline{\tau}_n -
                         b^{p^{n+1}-p^n} \underline{\tau}_{n+1}&
   \text{ coefficient of $t^{p^{n}}$}
\end{align}
Now, \rref{eq:nv1} is just \rref{ero12}. Plugging \rref{eq:useful1} into
\rref{eq:nv2}, we have
\begin{align*}
  b^{p-1}\underline{\tau}_{1}  = \rho^{-2}(\sigma^{-2} b^{-1}(\overline{u} - u) - \tau_0).
\end{align*}
Plugging in \rref{erou} and inducting based on \rref{eq:nv3}, we get \rref{erou102}.

Now we shall prove that $\widehat{\tau}_n =  \underline{\tau}_n\overline{u}$. From
\begin{equation*}
\lambda(s) = s \otimes 1 + q \otimes (\tau_0 + b^{p-2} \widehat{\xi}_1) +
\sum_{n \geq 1} y^{p^{n-1}} \otimes \widehat{\tau}_n
\end{equation*}
and $\lambda(s) = \lambda(\nu u) = \lambda(\nu)(1 \otimes \overline{u})$, it
remains to verify
\begin{equation*}
 s \otimes 1 + q \otimes (\tau_0 + b^{p-2} \widehat{\xi}_1) =  \nu \otimes
 \overline{u} - q \otimes b^{p-2} \underline{\xi}_1 \overline{u} + x \otimes
(\tau_0 + b^{p-2} \widehat{\xi}_1) \overline{u}.
\end{equation*}
Using \rref{eq:xi1hat} and \rref{eq:useful2} ($\overline{u}$ is exterior), this is
equivalent to
\begin{equation*}
\nu \otimes (u-\overline{u}) + q \otimes
\sigma^{-2}b^{-1}(\overline{u}-u) + x \otimes b^{-1}\sigma^{-2}u\overline{u} = 0.
\end{equation*}
Plugging in the Borel cohomology expressions, the left hand side is the sum of  
\begin{equation*}
z \otimes [\sigma^{-2}(u-\overline{u})+
\sigma^{-2}(\overline{u}-u)] = 0
\end{equation*}
and ($u$ is exterior)
\begin{equation*}
- \sigma^{-2}t \otimes
u \sigma^{-2}b^{-1}(\overline{u}-u) + \sigma^{-2}t \otimes b^{-1}\sigma^{-2}u\overline{u} = 0.
\end{equation*}
Thus, we have \rref{erou103}.
\end{proof}

\vspace{3mm}
From this, we can deduce further multiplicative relations
\beg{erou104}{
\widehat{\xi}_n\rho^{-2}=-\underline{\xi}_n\overline{u}\rho^{-2}
+b\underline{\tau}_n,
}
\beg{erou105}{
\widehat{\tau}_n\rho^{-2}=\underline{\tau}_n\overline{u}\rho^{-2},
}
\beg{erou106}{
b\widehat{\tau}_n=\widehat{\xi}_n\overline{u}\rho^{-2},
}
and
\beg{erou107}{
\widehat{\tau}_n\rho^{-2}\overline{u}=0.
}
Recall, of course, \rref{ero12}, and also \rref{eq:xi1hat}.

From this, using multiplication, division by $b$ in Tate cohomology, and total odd-dimensionality,
we deduce additional multiplicative relations
\beg{eroo1}{\widehat{\tau}_m\widehat{\tau}_n=\widehat{\tau}_n\underline{\tau}_n=\widehat{\tau}_n\widehat{\xi}_n=
\widehat{\xi}_n^2=\underline{\tau}_n^2=
0
}
\beg{eroo3}{
\widehat{\xi}_n\underline{\tau}_n=\underline{\xi}_n\widehat{\tau}_n.
}
Again, these relations are true on the nose, and not just in Borel cohomology, by Proposition \ref{pzpzp}
and the compatibility of the product and the coproduct. (Another even more general reason is given in the Comment
at the end of this section.)

\vspace{3mm}
Now using the fact that $H\underline{\Z/p}\wedge H\underline{\Z/p}$ is a wedge of $R$-suspensions of
$H\underline{\Z/p}$ and $HT$, monomials in $A_\star$ which are $b$-divisible in $A_\star^{cc}$
must also be $b$-divisible in $A_\star$. Thus, we also get elements of the form
\beg{edivdiv}{\frac{\widehat{\xi}_m\widehat{\xi}_n}{b},\; 
\frac{\widehat{\xi}_m\underline{\xi}_n-\widehat{\xi}_n\underline{\xi}_m}{b},\;
\frac{\widehat{\xi}_m\widehat{\tau}_n}{b}= \frac{\widehat{\tau}_m\widehat{\xi}_n}{b},
\frac{\widehat{\xi}_m\underline{\tau}_n-\underline{\xi}_m\widehat{\tau}_n}{b}
=\frac{\underline{\tau}_m\widehat{\xi}_n+\widehat{\tau}_m\underline{\xi}_n}{b},
}
(note that the last two are related by switching $m$ and $n$), and elements obtained by
iterating this procedure. 

In fact, this is related to the Example in Section \ref{prelim}. In the language of \cite{sw},
the ``quadruplet" $\widehat{\xi}_n,\underline{\xi}_n,\widehat{\tau}_n, \underline{\tau}_n$
generate an $HT_\star$ summand of $A_\star$. Thus, the elements \rref{edivdiv} are
precisely those divisions by $b$ which are allowed in $HT\wedge_{H\underline{\Z/p}}HT_\star$
according to the Example. We obtain the following 

\begin{proposition}\label{ppdiv}
To facilitate notation, put $d\widehat{\xi}_i=\underline{\xi}_i$, $d\widehat{\tau}_i=\underline{\tau}_i$ and
extend to products as a differential on a graded-commutative ring. For given
$i_1<\dots<i_k$, let $K(i_1,\dots,i_k)$ denote the sub-$\Z/p$-module of $A_\star$ spanned
by monomials of the form
\beg{ekkkk}{\kappa_{i_1}\dots\kappa_{i_k}\underline{\xi}_{i_1}^{s_1}\dots
\underline{\xi}_{i_k}^{s_k}}
where $0\leq s_\ell\leq p-2$, $\kappa_\ell$ can stand for any of the elements 
$\widehat{\xi}_\ell$, $\underline{\xi}_\ell$, $\widehat{\tau}_\ell$,
$\underline{\tau}_\ell$. Let $K_j(i_1,\dots,i_k)$ denote the submodule of $K(i_1,\dots,i_k)$
of elements divisible by $b^j$. Then $K_j(i_1,\dots,i_k)$ is spanned by elements
$$y,dy$$
where $y$ is of the form \rref{ekkkk} so that
at least $j+1$ of the elements $\kappa_{i_s}$ are of the form $\widehat{\xi}_{i_s}$ or $\widehat{\tau}_{i_s}$,
only at most one of them being $\widehat{\tau}_{i_s}$. Such $y$ will be called {\em admissible} (otherwise, $y=0$).
\end{proposition}

\begin{proof}
Note that the assertion of monomials which are not admissible being $0$ follows from \rref{eroo1}, \rref{eroo3}.
Additionally, the $b$-divisions claimed can be created inductively from \rref{edivdiv}.

Now from Proposition \ref{p44}, we obtain inductively 
$$HT^{\wedge_Hk}\simeq \bigvee_{j=0}^{k-1}{k-1\choose j}\Sigma^{j(\beta-1)}HT,$$
which shows that the listed divisions give all the required generators.
\end{proof}

\vspace{3mm}
Using the our computation of $H\underline{\Z/p}^\star B_{\Z/p}(\Z/p)$
in the last section, we similarly conclude that in Borel cohomology, we must have
\beg{emun1}{
\underline{\mu}_n=\frac{\widehat{\xi}_n\underline{\xi}_n^{p-1}+\underline{\theta}_n\rho^{-2}\overline{u}}{
b}-\widehat{\xi}_{n+1}b^{p-2}.
}
Note that the term $-\widehat{\xi}_{n+1}b^{p-2}$ in \rref{emun1} is represented by an element of $A_\star$ and could
be eliminated by changing the choice of $\omega$ in \rref{eomegaomega} by adding the term $b^{p-2}q$ to it.

By \cite{sw}, this element generates an $H\underline{\Z/p}_\star$-summand in $H\underline{\Z/p}_\star H\underline{\Z/p}$,
where the equation \rref{emun1} remains valid after multiplication by $b$. (Note that without this element, our calculations would predict
an $HM_\star$-summand multiplied by the generator $\underline{\theta}_n$. The equation \rref{emun1} shows that the presence
of $\underline{\mu}_n$ breaks this up into two shifted copies of $H\underline{\Z/p}_\star$.)

\vspace{3mm}

To prove \rref{emun1}, by \rref{ep102}, \rref{eomegaomega}, \rref{ecrosss}, and Proposition \ref{pzpzp}, we have
\begin{equation*}
b\omega = bz ((\sigma^{-2}t)^{p-1} - b^{p-1}) = (q+xu)(x^{p-1} -b^{p-1}) = q(x^{p-1} -b^{p-1}) + \sigma^{-2}yu
\end{equation*}
Applying $\lambda$ to this equation and plugging in $\lambda(\omega)$ to the left side and $\lambda(q)$, $\lambda(x)$ and $\lambda(y)$ to the right side, we get
\begin{align*}
  \omega \otimes b + y \otimes b\tau_0 & + \sum_{n \geq 1} y^{p^{n}} \otimes b \underline{\mu}_n+ \cdots\\
  & =  (q\otimes 1 +\sum_{n\geq 1} y^{p^{n-1}}\otimes \widehat{\xi}_{n})  (x\otimes 1+\sum_{n\geq 1}y^{p^{n-1}}\otimes\underline{\xi}_n)^{p-1} \\
  & - (q\otimes b^{p-1} +\sum_{n\geq 1} y^{p^{n-1}}\otimes \widehat{\xi}_{n}b^{p-1}) \\
& +(y\otimes\rho^{-2} \overline{u}+\sum_{n\geq 1}y^{p^n}\otimes\underline{\theta}_n \rho^{-2} \overline{u} )
\end{align*}
One can check that the coefficients of $y$ do match on both sides. Comparing coefficients of $y^{p^{n}}$, we have
\begin{equation*}
  b \underline{\mu}_n =  \widehat{\xi}_{n}(\underline{\xi}_n)^{p-1} -\widehat{\xi}_{n+1}b^{p-1} +\underline{\theta}_n \rho^{-2} \overline{u}.
\end{equation*}

\vspace{3mm}
\noindent
{\bf Comment:}
We note that in general, products relations in the Steenrod algebra
which hold in Borel cohomology cannot produce error terms in the derived wedge by the Tate diagram
$$
\diagram
(E\wedge E\Z/p_+)^{\Z/p}\dto_\equiv\rto &E^{\Z/p}\dto\rto &\Phi^{\Z/p}E\dto\\
(F(E\Z/p_+,E)\wedge E\Z/p_+)^{\Z/p}\rto^(.6)N & F(E\Z/p_+,E)^{\Z/p}\rto & \widehat{E}
\enddiagram
$$
whose right-hand square is a diagram of rings for a $\Z/p$-equivariant commutative ring spectrum $E$. (Note that in our case,
Borel cohomology injects into Tate cohomology.)

\vspace{3mm}

\section{The $\Z/p$-equivariant Steenrod algebra}\label{ssteen}

To express the additive structure of $A_\star$, we introduce a ``Cartan-Serre basis" of elements of the
form
\beg{eueu0}{
\tau_0^e\underline{\Theta}_S\underline{M}_Q,
}
and
\beg{eueu}{
\underline{\Xi}_R\underline{\Theta}_S\underline{\Tau}_E\tau_0^e
}
where 
\beg{eueu1}{\begin{array}{ll}
R=(r_1,r_2,\dots), & r_n\in \Z, 0\leq r_n<p,\\
Q=(q_1,q_2,\dots), & q_n\in \{0,1\},\\
S=(s_1,s_2,\dots), & s_n\in \Z, 0\leq s_n,\\
E=(e_1,e_2,\dots), & e_n\in \{0,1\},\\
e\in \{0,1\}&
\end{array}
}
As usual, only finitely many non-zero entries are allowed in each sequence.
Here we understand
$$\underline{\Theta}_S=\underline{\theta}_1^{s_1}\underline{\theta}_2^{s_2}\dots,$$
$$\underline{M}_Q=\underline{\mu}_1^{q_1}\underline{\mu}_2^{q_2}\dots,$$
$$\underline{\Xi}_R=\underline{\xi}_1^{r_1}\underline{\xi}_2^{r_2}\dots,$$
$$\underline{\Tau}_E=\underline{\tau}_1^{e_1}\underline{\tau}_2^{e_2}\dots\;.$$
The elements $\underline{\theta}_n$
have the dimensions introduced above. Additionally, we recall
$$|\underline{\xi}_n|=2p^{n-1}+(p^n-p^{n-1}-1)\beta,$$
$$|\underline{\tau}_n|=|\underline{\xi}_n|+1.$$

Additionally, for any sequence of natural numbers
$P$ with finitely many non-zero elements, we denote by $|P|$ the number of non-zero elements in $P$.
We will assume
$|R+E|\neq 0$ in \rref{eueu}.

\begin{theorem}\label{tfinal}
The dual $\Z/p$-equivariant Steenrod algebra $A_\star$ is, additively, a sum of copies of
$$H\underline{\Z/p}_\star$$
shifted by the total dimension of elements of the form \rref{eueu0} with $R=E=0$, and copies
of 
$$HM_\star$$
with admissible generators of the form $y, dy$ as in Proposition \ref{ppdiv} times $b^{-\ell}$ where
$\ell$ is the maximum number so that $y$ is divisible by $b^\ell$ according to Proposition \ref{ppdiv}.
The dimensions of admissible monomials in $\underline{\xi}_m$,
$\underline{\tau}_n$ are given by the values $|\underline{\Xi}_R|$, $|\underline{\Tau}_E|$ given above.

The good tail $A_\star^g$ is multiplicatively generated, as an algebra over 
$$\Z/p[\sigma^{-2},b]\otimes
\Lambda[\sigma^{-2}u]$$ 
by the elements $\underline{\xi}_n$, $\widehat{\xi}_n$, $\underline{\tau}_n$, $\widehat{\tau}_n$, 
$\underline{\theta}_n, \underline{\mu}_n$,
$n\geq 1$, $\tau_0$, subject to all relations valid in $A^{cc}_\star$, including
\rref{erou104}-\rref{eroo3}. The derived tail is given by
$$A_\star^d=H^2_{(b,\sigma^{-2})}(H\underline{\Z/p}_\star^g, A_\star^g)[-1]$$
and $A_\star$ is an abelian ring with respect to $A_\star^g$.
\end{theorem}

\begin{proof}
We already checked this on Borel (co)homology. On geometric fixed points (which are determined
by the ``ideal" part), the elements $\underline{\xi}_R$ match non-negative powers of $\rho^{-2}$,
$\tau_0$ matches $\overline{u}\rho^{-2}$. The elements $\widehat{\xi}_n, \widehat{\tau}_n,\underline{\tau}_n$
and their $\Z/p[\sigma^{-2}]\otimes\Lambda[u\sigma^{-2}]$-multiples
represent the $\Z/p[\sigma^{-2}]\otimes\Lambda[u\sigma^{-2}]$-multiples of $\tau_{n-1}$. The relations
guarantee that no element is represented twice. Thus, we  can deduce from the Tate diagram that the quotient by our
relations give the correct fixed points. (See  also the comments at the end of Section \ref{sborel}.) We also note that the 
answer we got agrees with the additive result of \cite{sw}.
\end{proof}

\end{document}